\documentclass{scrartcl}

\usepackage{amsmath}
\usepackage{amsfonts}
\usepackage{amsthm}
\usepackage{comment}
\usepackage{enumerate}
\usepackage{thmtools}
\usepackage{xcolor}

\usepackage{tikz}
\usetikzlibrary{calc}
\usetikzlibrary{shapes.geometric}
\tikzset{
	vertex/.style={fill, circle, inner sep = 1.2pt},
	leg/.style={draw, circle, inner sep = 1.2pt},
	component/.style={draw, minimum width = 7mm, minimum height = 7mm},
	empty triangle/.style = {regular polygon, regular polygon sides = 3, inner sep = -1pt, minimum width = 9.5mm},
	triangle/.style = {empty triangle, draw}
}

\usepackage[colorlinks=true]{hyperref}
\hypersetup{
	linkcolor = {blue},
	citecolor = {blue}
}
\usepackage[capitalize, noabbrev]{cleveref}

\newtheorem{theorem}{Theorem}[section]
\newtheorem{lemma}[theorem]{Lemma}

\newtheorem{observation}[theorem]{Observation}

\newcommand{\N}{\mathbb{N}}
\newcommand{\fA}{\mathcal{A}}

\newcommand{\fO}{\mathcal{O}}

\newcommand{\fX}{\mathcal{X}}

\newcommand{\pCleanup}{\texttt{cleanup}}

\newcommand{\td}{\mathrm{td}}

\newcommand{\cost}{\mathrm{cost}}
\newcommand{\depth}{\mathrm{depth}}

\newcommand{\alt}{\mathrm{alt}}

\newcommand{\BST}{\mathrm{bst}}
\newcommand{\OPT}{\mathrm{OPT}}
\newcommand{\OPTST}{\text{OPT-ST}}

\newcommand{\restr}[1]{|_{#1}}

\title{The diameter of caterpillar associahedra}
\author{Benjamin Aram Berendsohn\thanks{Institut f\"ur Informatik, Freie Universit\"at Berlin, \texttt{beab@zedat.fu-berlin.de}. Work supported by DFG grant KO 6140/1-1.}}

\begin{document}
	
	\maketitle
	
	\begin{abstract}
		The caterpillar associahedron $\fA(G)$ is a polytope arising from the rotation graph of search trees on a caterpillar tree $G$, generalizing the rotation graph of binary search trees (BSTs) and thus the conventional associahedron. We show that the diameter of $\fA(G)$ is $\Theta(n + m \cdot (H+1))$, where $n$ is the number of vertices, $m$ is the number of leaves, and $H$ is the entropy of the leaf distribution of $G$.
		
		Our proofs reveal a strong connection between caterpillar associahedra and searching in BSTs. We prove the lower bound using Wilber's first lower bound for dynamic BSTs, and the upper bound by reducing the problem to searching in static BSTs.
	\end{abstract}
	
	\section{Introduction}
	
	Associahedra are polytopes with importance in many areas of combinatorics. Among other things, the skeleton of the $(n-1)$-dimensional associahedron $\fA_n$ represents the \emph{rotation graph} of binary search trees (BSTs) on $n$ keys. More precisely, each vertex of $\fA_n$ represents a BST, and each edge represents a rotation (definitions of BSTs and rotations can be found in standard textbooks). The diameter of $\fA_n$ is known to be precisely $2n-6$ when $n > 10$~\cite{SleatorEtAl1988,Pournin2014}; that is, any BST can be transformed into any other BST with at most $2n-6$ rotations, and this is tight.
	
	In this paper, we study a generalization of associahedra called \emph{graph associahedra}. Graph associahedra were originally defined using \emph{tubings}~\cite{CarrDevadoss2004}. We use an equivalent definition based on \emph{search trees on graphs} (STGs).
	While the keyspace of a BST is a linearly ordered set, the key space of an STG is a graph. Formally, given a connected graph $G = (V,E)$, an STG on $G$ is a rooted tree that can be constructed as follows. Choose a vertex $r \in V$ as the root. Then, recursively create STGs on the connected components of $G \setminus v$, and add them to the STG as children of $r$. Rotations on STGs can be defined similarly as for BSTs (more details in \cref{sec:prelims}). Search trees on graphs have been used in various contexts under different names (see, e.g., \cite[Section 2.2]{CardinalEtAl2018b}).
	
	Given a connected graph $G$ on $n$ vertices, Carr and Devadoss~\cite{CarrDevadoss2004} defined the graph associahedron $\fA(G)$ as an $(n-1)$-dimensional polytope such that the skeleton of $\fA(G)$ is isomorphic to the rotation graph of the STGs on $G$. Since STGs on the path with $n$ vertices correspond to BSTs on $n$ nodes, we obtain the conventional associahedron when $G$ is a path.

	For search trees $T_1$ and $T_2$ on a graph $G$, let the \emph{rotation distance} $d(T_1, T_2)$ be the minimum number of rotations required to transform $T_1$ into $T_2$. The diameter $\delta(\fA(G))$ of $\fA(G)$ is the maximum rotation distance between two search trees on $G$.
	Manneville and Pilaud~\cite{MannevillePilaud2015} showed that $\max\{2n-18, m\} \le \delta(\fA(G)) \le \binom{n}{2}$ for each connected graph $G$ on $n$ vertices and $m$ edges. Moreover, the diameter of graph associahedra is monotone under the addition of edges. Both bounds are asymptotically tight. For example, conventional associahedra ($G$ is a path) and cyclohedra ($G$ is a cycle) have linear diameter, and permutohedra ($G$ a the complete graph) have diameter $\binom{n}{2}$.
	
	In this paper, we consider the case where $G$ is a \emph{caterpillar tree}. A caterpillar tree (or simply \emph{caterpillar}) is a tree consisting of a path and some number of leaves that are adjacent to the path.
	The choice of that path is not unique\footnote{Either the ends of the path are leaves, or there is a leaf attached to an end of the path which could be considered part of the path.}, but we assume that any considered caterpillar consists of a distinguished path called the \emph{spine} and any number of leaves, called \emph{legs}. Our results are not significantly affected by the choice of the spine.
	
	We determine the diameter of every caterpillar associahedron up to a constant factor. This involves the Shannon entropy of the ``leg distribution'', which we now properly define. Let $G$ be a caterpillar tree with $n$ spine vertices $s_1, s_2, \dots, s_n$, let $s_i$ be adjacent to $m_i$ leg vertices, and let $m = m_1 + m_2 + \dots + m_n$ be the total number of leg vertices. Then
	\begin{align*}
		H(G) = H(m_1, m_2, \dots, m_n) = \sum_{i \in [n], m_i > 0} \frac{m_i}{m} \log\left(\frac{m}{m_i}\right).
	\end{align*}
	
	For simplicity of presentation, we write $H'(\cdot) = H(\cdot) + 1$. We are now ready to state our main result.
	\begin{restatable}{theorem}{restateCatLB}\label{p:main}
		Let $G$ be a caterpillar tree with $n$ spine vertices and $m$ leg vertices. Then $\delta(\fA(G)) \in \Theta( n + m \cdot H'(G) )$.
	\end{restatable}
	
	Notably, if $m = n$ and each spine node is adjacent to one leaf node, than $\delta(\fA(G)) \in \Omega(n \log n)$.
	
	Our proofs make use of techniques from the design of \emph{optimal BSTs}. A connection between rotations in BSTs and rotations in search trees on caterpillars is not surprising -- caterpillars are similar to paths, after all. However, we show a connection to \emph{queries} to BSTs. Essentially, the leg nodes in search trees on caterpillars can be seen as queries to the BST on the spine nodes. For our upper bound (\cref{sec:upper}), we use the fact that an optimal \emph{static} BST for an input distribution $X$ has amortized query cost $H'(X)$~\cite{Mehlhorn1975}. For our lower bound (\cref{sec:lower}), we use \emph{Wilber's first lower bound} \cite{Wilber1989}, which bounds the performance of \emph{dynamic} BSTs on a certain input sequence. We show that it also bounds the rotation distance between certain search trees on a caterpillar. Finally, we show that Wilber's first lower bound is asymptotically equal to $H'(X)$ if the input distribution $X$ is fixed, but the order of queries is worst possible. Note that this also implies that dynamic BSTs cannot beat optimal static BST on any distribution if the ordering is worst possible. Kujala and Elomaa~\cite{KujalaElomaa2008} previously showed that this is true even if the ordering is random, but they did not use Wilber's bound.
	
	\paragraph{Related work.}
	Improved bounds on $\delta(\fA(G))$ are known if $G$ belongs to certain graph classes. Pournin~\cite{Pournin2014a} showed that $\delta(\fA(G)) \approx 2.5n$ if $G$ is the cycle on $n$ vertices. Cardinal, Langerman and Pérez-Lantero showed that $\delta(\fA(G)) \in \fO(n \log n)$ if $G$ is a tree on $n$ vertices, and this bound is tight if $G$ has the form of a balanced binary tree.
	
	Recently, Cardinal, Pournin, and Valencia-Pabon~\cite{CardinalEtAl2021} showed that $\delta(\fA(G)) \in \fO( \td(G) \cdot n )$, where $\td(G)$ is the \emph{treedepth}\footnote{The treedepth $\td(G)$ can be defined as the minimum height of a search tree on $G$.} of $G$, and that this bound is attained by \emph{trivially perfect graphs}. Using the relationship between treedepth and \emph{treewidth},
	this extends the $\fO(n \log n)$ upper bound to graphs with bounded treewidth. They also showed that this bound is is tight for graphs of \emph{pathwidth two} (which have treewidth at most two, but are not necessarily trees). For the definitions of treewidth and pathwidth, we refer to \cite{CardinalEtAl2021}. Our \cref{p:main} shows that the $\fO(n \log n)$ bound is tight already for caterpillars, which are both trees and have pathwidth \emph{one} (in fact, caterpillars are precisely the graphs of pathwidth one).
	
	We do not consider \emph{queries} to STGs in this paper. In the case where $G$ is a tree, some results from BSTs have been carried over. Bose, Cardinal, Iacono, Koumoutsos, and Langerman~\cite{BoseEtAl2019} presented a $\fO(\log \log n)$-competitive search tree algorithm based on \emph{tango trees} for BSTs~\cite{DemaineEtAl2007}. Berendsohn and Kozma~\cite{BerendsohnKozma2022} described a variant of Splay trees \cite{SleatorTarjan1985}, and a polynomial-time approximation scheme for the optimal static search tree on a given tree for a given input distribution. Notably, it is still unknown whether an optimal static search tree on a tree can be found in polynomial time.
	
	Berendsohn and Kozma~\cite{BerendsohnKozma2022} also showed that if we only consider a subset of search trees on a tree $G$ called $k$-cut trees, then the maximum rotation distance between two STGs is linear. A special case of $k$-cut trees are \emph{Steiner-closed trees}, which play a central role in the results of \cite{BoseEtAl2019} and \cite{BerendsohnKozma2022}.
	
	\section{Preliminaries}\label{sec:prelims}
	
	In this paper, we consider (simple and undirected) graphs on the one hand, and (rooted) search trees on the other. We call the vertices of search trees \emph{nodes}. In both cases, we denote by $V(\cdot)$ the set of vertices or nodes and by $E(\cdot)$ the set of edges.
	
	Let $G$ be a graph. We denote the subgraph of $G$ induced by $U \subseteq V(G)$ by $G[U]$. For $v \in V(G)$, we write $G \setminus v = G[V(G) \setminus\{v\}]$.
	
	Let $T$ be a rooted tree and $x \in V(T)$. For a node $x$, $T_x$ denotes the subtree of $T$ consisting of $x$ and all its descendants. The \emph{depth} of $x$ is the number of nodes in the path from the root of $T$ to $x$, and is denoted by $\depth_T(x)$.
	
	\paragraph{Queries in binary search trees.}
	
	In the \emph{dynamic BST model}, we are given a starting BST $S$ on $[n]$ and a sequence $\sigma$ of access queries. Each access query specifies a node $i \in [n]$. We start each query with a pointer at the root, and are required to move the pointer to the node $i$ to satisfy the query. To this end, we are allowed to move the pointer to the the parent or a child of the node it is currently pointing at, or execute a rotation involving that node. Let $\OPT(S, \sigma)$ denote the minimum number of pointer moves and rotation needed to serve $\sigma$. We charge a pointer move at the start of each query, when the pointer is moved to the root, so each query has cost at least one.
	
	Since the rotation distance between two BSTs is $\fO(n)$, we can always replace the starting BST $S$ by a different one at the cost of $\fO(n)$. If the access sequence is long enough, this cost is insignificant; therefore, let us define $\OPT(\sigma) = \min_S \OPT(S, \sigma)$. For each BST $S$, we have $\OPT(S, \sigma) \le \OPT(\sigma) + \fO(n)$.
	
	It is not known how to compute or approximate $\OPT(S, \sigma)$ or the associated sequence of operations efficiently. However, a number of algorithms have been conjectured to be \emph{instance-optimal}, i.e, to serve every access sequence $\sigma$ with a cost of $\fO(\OPT(\sigma) + n)$, most notably \emph{Splay}~\cite{SleatorTarjan1985} and \emph{Greedy}~\cite{Lucas1988,Munro2000}. We emphasize that Splay is an \emph{online} algorithm, i.e., it serves each query independently from future queries, and that Greedy can be made online~\cite{DemaineEtAl2009} with only a constant-factor overhead. It is currently unknown whether any online algorithm can approximate the offline optimum $\OPT(\sigma)$ by a constant factor; this is the subject of the \emph{dynamic optimality conjecture}.
	
	There are several lower bounds known for $\OPT(\sigma)$. In this paper, we use \emph{Wilber's first lower bound}~\cite{Wilber1989}, which we define and discuss in \cref{sec:wilber_bst}.
	
	If we do not allow rotations, then we can only move the pointer down until we hit the queried node. Thus, the minimum cost of serving $\sigma = (x_1, x_2, \dots, x_m)$ in $S$ is clearly $\sum_{i=1}^m \depth_S(x_i)$. A BST $S$ minimizing this quantity is called an \emph{optimal static BST}. Note that only the frequencies of the elements in $\sigma$ affect the static cost, not the order. For a BST $S$ on $[n]$ and element frequencies $m_1, m_2, \dots, m_n \in \N_0$, define $$\cost(S, m_1, m_2, \dots, m_n) = \sum_{i=1}^n m_i \cdot \depth_S(i).$$ Let $\OPTST(m_1, m_2, \dots, m_n)$ be the minimum $\cost(S, m_1, m_2, \dots, m_n)$ over all possible BSTs $S$.
	
	It is possible to compute an optimal static BST in $\fO(n^2)$ time~\cite{Knuth1971}. Mehlhorn showed that $\frac{1}{m}\OPTST(m_1, m_2, \dots, m_n)$ is within a factor two from the Shannon entropy.
	
	\begin{lemma}[{\cite{Mehlhorn1975}}]\label{p:opt-st-entropy}
		Let $X = (m_1, m_2, \dots, m_n)$ be a sequence of nonnegative integers, and let $m = m_1 + m_2 + \dots + m_n$. Then
		\begin{align*}
			\frac{1}{2}H(X) \cdot m \le \OPTST(X) \le (2H(X) + 2) \cdot m = 2H'(X) \cdot m.
		\end{align*}
	\end{lemma}
	
	\paragraph{Search trees on graphs.}
	
	Let $G$ be a connected graph, and $T$ be a rooted tree, such that $V(T) = V(G)$. Let $r$ be root of $T$ and let $c_1, c_2, \dots, c_k$ be the children of $r$ in $T$. Then $T$ is a \emph{search tree on $G$} if
	\begin{enumerate}[(a)]
		\itemsep0pt
		\item $G \setminus r$ consists of precisely $k$ connected components $C_1, C_2, \dots, C_k$ such that $V(T_{c_i}) = V(C_i)$ for each $i \in [k]$; and
		\item $T_{c_i}$ is a search tree on $C_i$ for each $i \in [k]$.
	\end{enumerate}
	
	Let $T$ be a search tree on $G$, let $p \in V(G)$, let $c$ be a child of $p$ in $T$, and let $g$ be the parent of $p$ in $T$, if $p$ is not the root. A \emph{rotation} of the edge $(p,c)$ makes $c$ the parent of $p$ and child of $g$ (or root), and accordingly redistributes children so that the result is still an STG. More precisely, it (1) makes $c$ a child of $g$, if $p$ is not the root, and otherwise, makes $c$ the root; (2) makes $p$ a child of $c$; and (3) makes each child $x$ of $c$ a child of $p$ where $V(T_c)$ contains both a vertex adjacent to $p$ and a vertex adjacent to $c$. See \cref{fig:rotation} for an illustration. It can be checked that the rooted tree resulting from a rotation is indeed an STG. It is also easy to see that each STG on $G$ can be rotated into every other STG on $G$, e.g., by rotating the correct element to the root and then recursing on the subtrees.
	
	\begin{figure}
		\centering
		\begin{tikzpicture}[xscale=1, yscale=0.75]
			\small
			\node[component] (A1) at (0,0) {$A_j$};
			\node[] (ARem) at (0,1+0.1) {$\vdots$};
			\node[component] (A2) at (0,2) {$A_1$};
			\node[vertex] (p) at (1,1) {};
			\node[above=1mm] at (p) {$p$};
			\node[component] (B1) at (2,0) {$B_k$};
			\node[] (BRem) at (2,1+0.1) {$\vdots$};
			\node[component] (B2) at (2,2) {$B_1$};
			\node[vertex] (c) at (3,1) {};
			\node[above=1mm] at (c) {$c$};
			\node[component] (C1) at (4,0) {$C_\ell$};
			\node[] (CRem) at (4,1+0.1) {$\vdots$};
			\node[component] (C2) at (4,2) {$C_1$};
			
			\draw[] (A1) -- (p) -- (B1) -- (c) -- (C1);
			\draw[] (A2) -- (p) -- (B2) -- (c) -- (C2);
		\end{tikzpicture}
		\hspace{5mm}
		\begin{tikzpicture}[sibling distance = 9mm, level distance = 13mm, child anchor = north]
			\footnotesize
			\node[vertex] (p) {}
				child{ node[triangle] (A1) {\strut$A_1$} }
				child{ node[triangle] (A2) {\strut$A_j$} }
				child{ node[vertex] (c) {}
					child{ node[triangle] (B1) {\strut$B_1$} }
					child{ node[triangle] (B2) {\strut$B_k$} }
					child{ node[triangle] (C1) {\strut$C_1$} }
					child{ node[triangle] (C2) {\strut$C_\ell$} }
				};
			\draw (p) -- ($(p)+(0,0.3)$);
			\node[right=0.5mm] at (p) {$p$};
			\node[above right] at (c) {$c$};
			\node at ($(A1)!0.5!(A2)+(0,0.2)$) {$\dots$};
			\node at ($(B1)!0.5!(B2)+(0,0.2)$) {$\dots$};
			\node at ($(C1)!0.5!(C2)+(0,0.2)$) {$\dots$};
		\end{tikzpicture}
		\hspace{5mm}
		\begin{tikzpicture}[sibling distance = 9mm, level distance = 13mm, child anchor = north]
			\footnotesize
			\node[vertex] (c) {}
				child{ node[vertex] (p) {}
					child{ node[triangle] (A1) {\strut$A_1$} }
					child{ node[triangle] (A2) {\strut$A_j$} }
					child{ node[triangle] (B1) {\strut$B_1$} }
					child{ node[triangle] (B2) {\strut$B_k$} }
				}
				child{ node[triangle] (C1) {\strut$C_1$} }
				child{ node[triangle] (C2) {\strut$C_\ell$} };
			\draw (c) -- ($(c)+(0,0.3)$);
			\node[left=0.5mm] at (c) {$c$};
			\node[above left] at (p) {$p$};
			\node at ($(A1)!0.5!(A2)+(0,0.2)$) {$\dots$};
			\node at ($(B1)!0.5!(B2)+(0,0.2)$) {$\dots$};
			\node at ($(C1)!0.5!(C2)+(0,0.2)$) {$\dots$};
		\end{tikzpicture}
		\caption{An STG rotation. \emph{(left)}~A graph $G$ with two vertices $p$ and $c$ that split $G$ into $j + k + \ell$ components. Each line represents one or more edges. \emph{(center)}~The subtree $T_p$ in an STG $T$ on $G$. \emph{(right)}~The result of the rotation $(p,c)$ in $T$.}\label{fig:rotation}
	\end{figure}
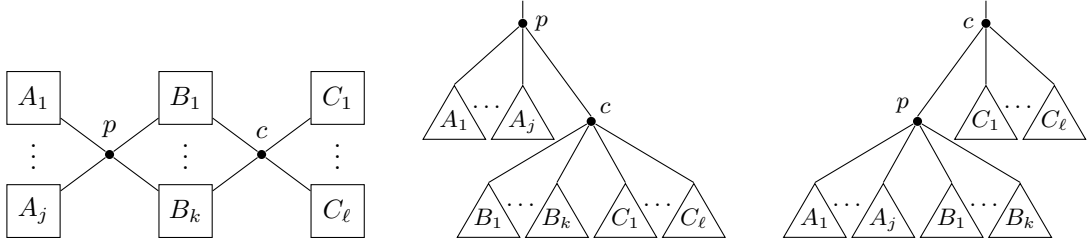
	
	\paragraph{Projections of STGs.} Both of our proofs make use of a concept defined by Cardinal, Langerman, and Pérez-Lantero \cite{CardinalEtAl2018b}.
	Let $T$ be a search tree on a graph $G$, and let $v$ be a leaf vertex in $G$. Note that $v$ has at most one child in $T$.
	We define $T \setminus v$ to be the following search tree on $G \setminus v$, obtained by \emph{pruning} $v$: If $v$ has no children, simply remove it. If $v$ has a parent $p$ and a child $c$, remove $v$ and make $c$ a child of $p$. If $v$ is the root of $T$ and has a child $c$, then remove $v$ and make $c$ the root.
	
	If $G$ is a tree, then we can obtain every subgraph of $G$ by progressively removing leaves. Accordingly, if $T$ is a search tree on a tree $G$, and $U \subseteq V(G)$ is a set of vertices such that $G[U]$ is connected, then we can define the \emph{projection} of $T$ onto $U$, written $T[U]$, as the search tree obtained by progressively pruning the vertices in $V(G) \setminus U$. It is easy to see that the order of pruning does not matter.
	
	The main utility of projections lies in the following lemma, which essentially states that projections onto $U$ are only affected by rotations between nodes in $U$.
	\begin{lemma}[{\cite{CardinalEtAl2018b}}]\label{p:proj}
		Let $T$ be a search tree on a tree $G$, let $U \subseteq V(G)$ such that $G[U]$ is connected, let $(x, y)$ be an edge of $T$, and let $T'$ be the tree obtained by rotating $(x,y)$. If $x, y \in U$, then $T'[U]$ is the STG obtained when rotating $(x,y)$ in $T[U]$. Otherwise, $T'[U] = T[U]$.
	\end{lemma}
	 
	\section{Search trees on caterpillars}\label{sec:structure}
	
	Let $n \in \N_+$ and $m_1, m_2, \dots, m_n \in \N_0$, and write $m = \sum_{i=1}^n m_i$. We define the caterpillar $C(m_1, m_2, \dots, m_n)$ with $n$ spine vertices and $m$ leg vertices as follows. The spine of $G$ consists of the vertices $s_1, s_2, \dots, s_n$, in that order. Additionally, for each $i \in [n]$ and $j \in [m_i]$, there is a leg vertex $\ell_{i,j}$ that is adjacent to $s_i$. Clearly, every caterpillar can be constructed this way.
	
	Let $T$ be a search tree on $G = C(m_1, m_2, \dots, m_n)$, and let $x \in V(T)$. We call $x$ a \emph{leg node} if it corresponds to a leg vertex, and a \emph{spine node} if it corresponds to a spine vertex. We denote nodes in $T$ in the same way we denote vertices in $G$, i.e., we write $\ell_{i,j}$ for leg nodes and $s_i$ for spine nodes.
	
	We call a leg node \emph{bound} if it has no children, and \emph{free} otherwise.
	
	\begin{observation}\label{p:stc-structure}
		Let $T$ be a search tree on $C(m_1, m_2, \dots, m_n)$. Consider a leg node $\ell_{i,j}$. If $\ell_{i,j}$ has no children, then $s_i$ is its parent. Otherwise, $\ell_{i,j}$ has exactly one child, and $s_i$ is a descendant of $\ell_{i,j}$.
	\end{observation}
	
	Define the \emph{spine BST} $\BST(T)$ as the projection of $T$ onto the spine vertices of $G$ (see \cref{fig:stcs} for an example). Since the spine vertices form a path, $\BST(T)$ indeed corresponds to a binary search tree. By \cref{p:proj}, each rotation between two spine nodes in $T$ corresponds to a BST rotation in $\BST(T)$. However, the converse is clearly not true in general, since two neighboring nodes $u, v$ in $\BST(T)$ might have leg nodes between them in $T$, in which case a rotation of $u,v$ in $BST(T)$ cannot be applied to $T$. Call an edge $(p,c)$ of $\BST(T)$ \emph{light} if $(p,c)$ is also an edge in $T$. Essentially, as long as we restrict ourselves to light edges, we can apply BST restructuring algorithms to $T$. This will be useful to prove our upper bound. We further observe that rotations between spine nodes do not affect the parents of leg nodes.
	
	\begin{observation}\label{p:spine-rot}
		Let $T$ be a search tree on a caterpillar, let $T'$ arise from a rotation between spine nodes in $T$, and let $\ell$ be a leg node. If $\ell$ is the root of $T$, then $\ell$ is also the root of $T'$. Otherwise, the parent of $\ell$ in $T'$ is the same as in $T$.
	\end{observation}
	
	\begin{figure}
		\newcommand{\normalsibdist}{6mm}
		\newcommand{\normallevdist}{5mm}
		\newcommand{\interpicdist}{3mm}
		\centering
		\begin{tikzpicture}[xscale=0.5,yscale=0.7]
			\small
			\node[vertex] (s1) at (0,0) {};
			\node[vertex] (s2) at (0,-1) {};
			\node[vertex] (s3) at (0,-2) {};
			\node[vertex] (s4) at (0,-3) {};
			\node[vertex] (s5) at (0,-4) {};
			\node[left] at (s1) {\strut$s_1$};
			\node[left] at (s2) {\strut$s_2$};
			\node[left] at (s3) {\strut$s_3$};
			\node[left] at (s4) {\strut$s_4$};
			\node[left] at (s5) {\strut$s_5$};
			\draw (s1) -- (s2) -- (s3) -- (s4) -- (s5);
			
			\node[leg] (l11) at (1,0.3) {};
			\node[leg] (l12) at (1,-0.3) {};
			\node[leg] (l31) at (1,-2) {};
			\node[leg] (l41) at (1,-3) {};
			\node[leg] (l51) at (1,-4+0.3) {};
			\node[leg] (l52) at (1,-4-0.3) {};
			
			\node[right] at (l11) {\strut$\ell_{1,1}$};
			\node[right] at (l12) {\strut$\ell_{1,2}$};
			\node[right] at (l31) {\strut$\ell_{3,1}$};
			\node[right] at (l41) {\strut$\ell_{4,1}$};
			\node[right] at (l51) {\strut$\ell_{5,1}$};
			\node[right] at (l52) {\strut$\ell_{5,2}$};
			\draw (l11) -- (s1) -- (l12);
			\draw (l31) -- (s3);
			\draw (l41) -- (s4);
			\draw (l51) -- (s5) -- (l52);
		\end{tikzpicture}
		\hspace{\interpicdist}
		\begin{tikzpicture}[
				sibling distance = \normalsibdist, level distance = \normallevdist,
				level 2/.style={sibling distance=2*\normalsibdist},
				level 3/.style={sibling distance=\normalsibdist}
			]
			\small
			\node[leg] (l31) {}
				child { node[vertex] (s2) {}
					child { node[leg] (l11) {}
						child { node[leg] (l12) {}
							child { node[vertex] (s1) {} }
						}
					}
					child { node[vertex] (s4) {}
						child { node[vertex] (s3) {} }
						child { node[leg] (l41) {} }
						child { node[leg] (l51) {}
							child { node[vertex] (s5) {}
								child { node[leg] (l52) {} }
							}
						}
					}
				};
			\node[left] at (s1) {\strut$s_1$};
			\node[right] at (s2) {\strut$s_2$};
			\node[below=-1mm] at (s3) {\strut$s_3$};
			\node[right] at (s4) {\strut$s_4$};
			\node[right] at (s5) {\strut$s_5$};
			\node[left] at (l11) {\strut$\ell_{1,1}$};
			\node[left] at (l12) {\strut$\ell_{1,2}$};
			\node[right] at (l31) {\strut$\ell_{3,1}$};
			\node[below=-1mm] at (l41) {\strut$\ell_{4,1}$};
			\node[right] at (l51) {\strut$\ell_{5,1}$};
			\node[right] at (l52) {\strut$\ell_{5,2}$};
		\end{tikzpicture}
		\hspace{\interpicdist}
		\begin{tikzpicture}[sibling distance = \normalsibdist, level distance = \normallevdist]
			\small
			\node[vertex] (s2) {}
				child { node[vertex] (s1) {} }
				child { node[vertex] (s4) {}
					child { node[vertex] (s3) {} }
					child { node[vertex] (s5) {} }
				};
			\node[left] at (s1) {\strut$s_1$};
			\node[right] at (s2) {\strut$s_2$};
			\node[left] at (s3) {\strut$s_3$};
			\node[right] at (s4) {\strut$s_4$};
			\node[right] at (s5) {\strut$s_5$};
			\node at (0, -2.5) {}; 
		\end{tikzpicture}
		\hspace{\interpicdist}
		\begin{tikzpicture}[sibling distance = \normalsibdist,
				level 1/.style={level distance = 3.5mm},
				level 7/.style={level distance = \normallevdist}
			]
			\small
			
			\node[leg] (l51) {}
				child { node[leg] (l41) {}
					child { node[leg] (l11) {}
						child { node[leg] (l52) {}
							child { node[leg] (l12) {}
								child { node[leg] (l31) {}
									child { node[vertex] (s2) {}
										child { node[vertex] (s1) {} }
										child { node[vertex] (s4) {}
											child { node[vertex] (s3) {} }
											child { node[vertex] (s5) {} }
										}
									}
								}
							}
						}
					}
				};
			\node[left] at (s1) {\strut$s_1$};
			\node[right] at (s2) {\strut$s_2$};
			\node[below=-1mm] at (s3) {\strut$s_3$};
			\node[right] at (s4) {\strut$s_4$};
			\node[right] at (s5) {\strut$s_5$};
			\node[right] at (l11) {\strut$\ell_{1,1}$};
			\node[right] at (l12) {\strut$\ell_{1,2}$};
			\node[right] at (l31) {\strut$\ell_{3,1}$};
			\node[right] at (l41) {\strut$\ell_{4,1}$};
			\node[right] at (l51) {\strut$\ell_{5,1}$};
			\node[right] at (l52) {\strut$\ell_{5,2}$};
		\end{tikzpicture}
		\hspace{\interpicdist}
		\begin{tikzpicture}[sibling distance = \normalsibdist, level distance = \normallevdist,
				level 1/.style={sibling distance=2*\normalsibdist},
				level 2/.style={sibling distance=\normalsibdist}
			]
			\small
			\node[vertex] (s2) {}
				child { node[vertex] (s1) {}
					child { node[leg] (l11) {} }
					child { node[leg] (l12) {} }
				}
				child { node[vertex] (s4) {}
					child { node[vertex] (s3) {}
						child { node[leg] (l31) {} }
					}
					child { node[leg] (l41) {} }
					child { node[vertex] (s5) {}
						child { node[leg] (l51) {} }
						child { node[leg] (l52) {} }
					}
			};
			\node[left] at (s1) {\strut$s_1$};
			\node[right] at (s2) {\strut$s_2$};
			\node[above] at (s3) {\strut$s_3$};
			\node[right] at (s4) {\strut$s_4$};
			\node[right] at (s5) {\strut$s_5$};
			\node[below=-1mm] at (l11) {\strut$\ell_{1,1}$};
			\node[below=-1mm] at (l12) {\strut$\ell_{1,2}$};
			\node[below=-1mm] at (l31) {\strut$\ell_{3,1}$};
			\node[below=-1mm] at (l41) {\strut$\ell_{4,1}$};
			\node[below=-1mm] at (l51) {\strut$\ell_{5,1}$};
			\node[below=-1mm] at (l52) {\strut$\ell_{5,2}$};
		\end{tikzpicture}
		\caption{\emph{(left)} The caterpillar $G = C(2,0,1,1,2)$.
			\emph{(center left)}~An STG $T$ on $G$.
			\emph{(center)}~The spine BST $S = \BST(T)$.
			\emph{(center right)}~The STG $A( S, \pi)$ with $\pi = (\ell_{3,1}, \ell_{1,2}, \ell_{5,2}, \ell_{1,1}, \ell_{4,1}, \ell_{5,1})$.
			\emph{(right)}~The STG $B(S)$.}\label{fig:stcs}
	\end{figure}
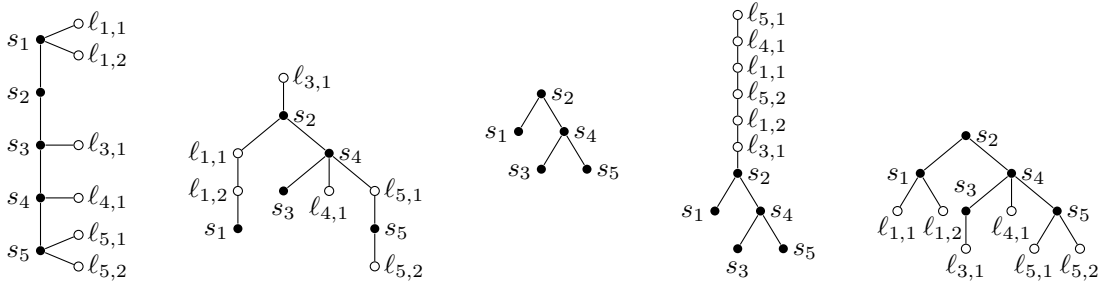

	\paragraph{Special STGs.}
	
	Before proceeding with the proofs, we define two useful kinds of STGs where the spine BST has only light edges. Let $G = C(m_1, m_2, \dots, m_n)$ be a caterpillar, let $S$ be a BST on the spine nodes of $G$, and let $\pi$ be an ordering of the leg nodes of $G$. Define $A(S, \pi)$ to be the unique search tree on $G$ such that $\BST(A(S, \pi)) = S$, each leg node is above each spine node (i.e., the leg nodes form a path at the top of the tree), and the order of the leg nodes from \emph{bottom to top} is $\pi$. Define $B(S)$ to be the unique search tree on $G$ such that all leg nodes are bound and $\BST(B(S)) = S$. Clearly, every search tree $T$ on $G$ without free leg nodes is equal to $B(\BST(T))$. \Cref{fig:stcs} shows examples.

	\section{Upper bound}\label{sec:upper}
	
	Fix a caterpillar $G = C(m_1, m_2, \dots, m_n)$. We first consider only search trees without free leg nodes.
	
	\begin{lemma}\label{p:upper-all-bound}
		Let $T_1, T_2$ be STGs on $G$ without free leg nodes. Then, $T_1$ can be transformed into $T_2$ with $\fO(n)$ rotations.
	\end{lemma}
	\begin{proof}
		Let $S_1 = \BST(T_1)$ and $S_2 = \BST(T_2)$. As stated in the introduction, there is a sequence of at most $\fO(n)$ rotations that transforms $S_1$ into $S_2$. We simply apply these rotations to $T_1$. For this to be well-defined, we need to show that we never (attempt to) rotate a heavy edge. Since $T_1$ has no free leg nodes, all edges in $S_1$ are light, so the first rotation goes through. Furthermore, a rotation between two spine nodes can never change a leg node from bound to free (or vice versa). Thus, by induction, after each rotation, all leg nodes are still bound, so we can apply the next rotation.
	\end{proof}
	
	The above lemma provides us with a ``core'' of the caterpillar associahedron with linear diameter. In the following, we show that the rotation distance from any search tree to \emph{some} STG without free leg nodes is $\fO(n + m H'(G))$. By the triangle inequality, this means that the rotation distance between any two STGs is at most $2 \cdot \fO(n + m H'(G)) + \fO(n) = \fO(n + m H'(G))$, and thus we have the upper bound of \cref{p:main}.
	
	We first show how to reduce the problem to the case that $T = A(S, \pi)$ for some BST $S$ and some leg node ordering $\pi$. For this, the following lemma is useful.
	
	\begin{lemma}\label{p:bst-restr-rot}
		Let $T$ be a BST. There exists a sequence of $\fO(n)$ rotations on $T$ such that
		\begin{enumerate}[(i)]
			\itemsep0pt
			\item every rotation involves only nodes at depth at most 3; and
			\item every spine node becomes the root of $\BST(T)$ at some point.
		\end{enumerate}
	\end{lemma}
	\begin{proof}
		 We start by rotating $T$ into the \emph{right path}, where no node has a left child. Cleary~\cite{Cleary2002} showed that this is possible using $\fO(n)$ rotations at the root or its right child. Then, repeatedly rotate the root with its right child, until the root has no right child. This way, each node becomes the root at least once. Clearly, the rotations used only involve the root, its children, and its grandchildren.
	\end{proof}
	
	\begin{lemma}\label{p:upper-reduction}
		Let $T$ be an arbitrary search tree on $G$. Then $T$ can be transformed into some $A(S,\pi)$ using $2m + \fO(n)$ rotations.
	\end{lemma}
	\begin{proof}
		An STG has the form $A(S,\pi)$ if and only if each leg node has no spine ancestor. The basic idea of the proof is to apply rotations between spine nodes to eventually bring each spine node to the root (of the spine BST). At any point, if the current root of the spine BST has a leg node as a child, rotate it with the leg node. We refer to all such rotations between two spine node rotations as the $\pCleanup$ step. By \cref{p:spine-rot}, every leg node is transported to the top this way.
		
		The problem with this approach is that we can only rotate light edges in the spine BST, and the only edges that we know must be light are the edges between the BST root and its children. However, if we extend our $\pCleanup$ step to consider leg nodes that are somewhat deeper in the tree (that is, nodes with two spine ancestors instead of just one), we guarantee that all BST edges \emph{near} the BST root are light. This allows us to apply \cref{p:bst-restr-rot} to bring each spine node to the root. We now describe the sequence of rotations more formally, starting with the $\pCleanup$ step.
		
		Let $T'$ be the current STG. Let $\pCleanup(T')$ be the following sequence of rotations: As long as there is a leg node $\ell$ with a spine parent $p$ and at most two spine ancestors, rotate $(p,\ell)$. Arbitrarily resolve conflicts. Let $T''$ be the STG after applying $\pCleanup$. Clearly, no spine node $s$ with $\depth_{\BST(T'')}(s) \le 2$ has a leg node child, so all edges in $\BST(T'')$ involving the root or its children are light. Moreover, each leg node that is touched by $\pCleanup$ is rotated at most twice, and afterwards has no spine node ancestors.
		
		Let $\fX$ be the sequence of rotations obtained by applying \cref{p:bst-restr-rot} to $\BST(T)$. We first apply $\pCleanup$ to $T$, then apply the spine rotations in $\fX$, with a $\pCleanup$ step after each spine rotation. Since each rotation in $\fX$ involves either the root of the spine BST or one of its children, the rotation is applied to a light edge. Thus, the whole sequence can indeed be applied to $T$.
		
		The number of rotations is at most $2m + \fO(n)$. Indeed, since no rotation between spine nodes can change the parent of a leg node (see \cref{p:spine-rot}), each leg node is only touched in a single $\pCleanup$ step (where it is rotated above all spine nodes), and only twice in that $\pCleanup$ step. The length of $\fX$ is $\fO(n)$ by \cref{p:bst-restr-rot}.
		
		Finally, we show that the final tree is indeed of the form $A(S, \pi)$. For this, it suffices to show that each leg node that has at least one spine ancestor in $T$ is touched in some $\pCleanup$ step. Suppose that such a leg node $\ell$ is not involved in a $\pCleanup$ step. Without loss of generality, let the parent $p$ of $\ell$ be a spine node. By \cref{p:spine-rot}, since $\ell$ is not touched in a $\pCleanup$ step and we never rotate between leg nodes, $p$ stays the parent of $\ell$ throughout the sequence of rotations. However, then $p$ will be the root of $\BST(T)$ at some point, so $\ell$ is rotated upwards by the next $\pCleanup$ step, a contradiction.
	\end{proof}
	
	It remains to show how to transform $A(S,\pi)$ into an STG without free leg nodes.
	
	\begin{lemma}\label{p:A_S_pi}
		Let $T = A(S,\pi)$. Then there is a sequence of $\fO(n + H'(G) \cdot m)$ rotations that transform $T$ into a search tree without bound leg nodes.
	\end{lemma}
	\begin{proof}
		Since every edge in $\BST(T)$ is light, we can first transform $\BST(T)$ into an arbitrary BST $S'$ using $\fO(n)$ rotations. We will later specify $S'$. Let $T' = A(S', \pi)$ be the resulting STG. Now pick the lowest leg node $\ell_{i,j}$ in $T'$, and rotate it down until it is bound (i.e., a child of $s_i$). Clearly, this requires $\depth_{\BST(T')}(s_i) = \depth_{S'}(s_i)$ rotations. Repeat this until all leg nodes are bound.
		
		The total number of rotations is
		\begin{align*}
			\fO(n) + \sum_{i=0}^n m_i \cdot \depth_{S'}(s_i).
		\end{align*}
	
		This is precisely $\cost(S', m_1, m_2, \dots, m_n)$, the cost of \emph{accessing} each $i \in [n]$ with frequency $m_i$ in the static BST $S'$. As such, if we choose $S'$ to be the optimal static BST for these frequencies, we need $\fO(n) + \OPTST(m_1, m_2, \dots, m_n) \le \fO(n) + 2 H'(G) \cdot m$ rotations, by \cref{p:opt-st-entropy}.
	\end{proof}
	
	\Cref{p:upper-all-bound,p:upper-reduction,p:A_S_pi} together imply the upper bound in \cref{p:main}.
	
	In the proof of \cref{p:A_S_pi}, we essentially treat the leg nodes as queries to our optimal static BST, where a leg node $\ell_{i,j}$ queries the spine node $s_i$. Rotating the leg nodes down is akin to moving down the pointer in the static BST model. Here, the pointer always points at the parent of the one leg node that has a spine node parent.
	
	Observe that we can similarly implement the dynamic BST model as rotations transforming $A(S, \pi)$ into a search tree without bound leg nodes, simply by allowing spine node rotations (BST rotations) in between leg node rotations (pointer moves). If the dynamic BST algorithm wants to rotate the single heavy edge in the spine BST of our STG, we have to move the leg node out of the way (and back afterwards), but this only adds a constant-factor overhead. Thus we obtain a generalization of the dynamic BST model, where we can start processing queries before finishing previous ones (although the way ``pointers'' work in this model is not very intuitive).
	
	Let $\sigma = \sigma(\pi)$ be the sequence of spine nodes obtained by replacing every leg node in $\pi$ by its adjacent spine node.
	Our observations imply that transforming $A(S,\pi)$ into an STG without free leg nodes requires no more than $\fO(\OPT(S, \sigma))$ rotations, and \cref{p:A_S_pi} essentially uses the fact that $\OPT(S, \sigma) \le \OPTST(m_1, m_2, \dots, m_n)$. In the next section, we show that \emph{Wilber's first lower bound}~\cite{Wilber1989} for $\OPT(S, \sigma)$ also holds for our generalized model.
	
	\section{Lower bound}\label{sec:lower}
	
	We start by defining a variant of Wilber's first lower bound and proving that it is equal to the Shannon entropy of the query distribution in the worst case (up to a constant factor). Then, we show that it also bounds the rotation distance between $A(S, \sigma)$ and $B(S)$ if $\sigma$ is the worst-case ordering and $S$ is a suitable search tree.
	
	\subsection{Wilber's first lower bound for binary search trees}\label{sec:wilber_bst}
	
	Let $S$ be a binary search tree on $n$ nodes, let $\sigma = (x_1, x_2, \dots, x_m)$ be a sequence of queries, and let $u$ be a node of $S$. Then we define $\lambda(S, u, \sigma)$ as follows. If $u$ has at most one child, then $\lambda(S,u,\sigma) = 0$. Otherwise, let $v, w$ be the children of $u$ and write $A = V(S_u) = \{u\} \cup V(S_v) \cup V(S_w)$. Let $\sigma\restr{A}$ be the sequence obtained from $\sigma$ by removing all elements not in $A$. Now $\lambda(S,u,\sigma)$ is defined as the number of times the sequence $\sigma\restr{A}$ switches between an element of $V(S_v)$, an element of $V(S_w)$, and $u$. More formally, $\lambda(S,u,\sigma)$ is the number of pairs of adjacent values $x,y$ in $\sigma$ such that neither $x,y \in V(S_v)$, nor $x,y \in V(S_w)$, nor $x = y = u$. Let $\Lambda(S,\sigma) = \sum_{u \in V(S)} \lambda(S,u,\sigma)$.
	
	For convenience, define $\lambda'(S,u,\sigma)$ as $\lambda(S,u,\sigma)$ plus the number of occurrences of $u$ in $\sigma$, and let $\Lambda'(S,\sigma) = \sum_{u \in V(S)} \lambda'(S,u,\sigma) = \Lambda(S,\sigma) + m$.
	It is known that $\OPT(S, \sigma) \in \Omega(\Lambda'(S, \sigma))$. This is not tight in general~\cite{ChalermsookEtAl2020,LecomteWeinstein2020}.
	Still, Wilber~\cite{Wilber1989} showed that if $\sigma$ is the \emph{bit reversal permutation}, then $\Lambda'(S, \sigma) \in \Theta(n \log n)$ for all $S$. This bound is already matched by a balanced static tree, so, on that sequence, Wilber's bound is tight and dynamic BSTs do not perform better than balanced trees. We now generalize this result to arbitrary distributions.
	
	\begin{lemma}\label{p:wilber-static-opt}
		Let $n \in \N_+$, let $m_1, m_2, \dots, m_n \in \N_0$, and let $m = \sum_{i=1}^n m_i$.
		Then there is a BST $S$ on $[n]$ and a sequence $\sigma$ of length $m$ where each $i \in [n]$ occurs precisely $m_i$ times, such that $\Lambda'(S,\sigma) \ge \frac{1}{2}H(m_1, m_2, \dots, m_n)$.
	\end{lemma}
	\begin{proof}
		We recursively construct a BST $S_{p,q}$ on the interval $[p,q]$, where $1 \le p \le q \le n$, and in the end set $S = S_{1,n}$. The construction is essentially the approximately optimal static BST construction due to Mehlhorn~\cite{Mehlhorn1975}.
		
		Fix $p$ and $q$. Let $k = q-p+1$ be the number of nodes in $S_{p,q}$, and, for each $i$ with $p \le i \le q$, let $a_i = \sum_{j=p}^{i-1} m_j$ and $b_i = \sum_{j=i+1}^q m_j$.
		We claim that there exists an $i \in [p,q]$ such that $m_i + \min(a_i,b_i) \ge \frac{k}{2}$. Suppose not. Then, for each $i$, we have either (1)~$m_i + a_i < \frac{k}{2} < b_i$ or (2)~$m_i + b_i < \frac{k}{2} < a_i$. Clearly, for $i = p$, (2) cannot hold, and likewise for $i = q$, (1) cannot hold. Let $i'$ be the highest index where (1) holds. Then, $i < q$ and $m_i + a_i < b_i = m_{i+1} + b_{i+1} < a_{i+1} = m_i + a_i$, a contradiction.
		
		Choose $i \in [p,q]$ such that $m_i + \min(a_i,b_i) \ge \frac{k}{2}$. Make $i$ the root of $S_{p,q}$, and attach the recursively constructed subtrees $S_{p,i-1}$ and $S_{i+1,q}$ as the left and right child to it (for $p' > q'$, we let $S_{p',q'}$ be the empty BST).
		
		Let $c(p,q) = \sum_{j=p}^q m_j \cdot \depth_{S_{p,q}}(j)$. Intuitively, $c(p,q)$ is the cost of accessing the relevant nodes within $S_{p,q}$. We now recursively construct a sequence $\sigma_{p,q}$ such that $c(p,q) \le 2 \Lambda'(S_{p,q}, \sigma_{p,q})$.
		
		First, if $p = q$, then let $\sigma_{p,q}$ simply consist of $m_p$ times the element $p$. Clearly, $c(p,q) = m_p = \Lambda'(S_{p,q}, \sigma_{p,q})$. Otherwise, let $i$ be the root of $S_{p,q}$. We construct $\sigma_{p,q}$ by combining $\sigma_{p,i-1}$, $\sigma_{i+1,q}$, and the $m_i$ occurrences of the element $i$ as follows. Start with the $m_i$ occurrences of $i$, then alternate between $\sigma_{p,i-1}$ and $\sigma_{i+1,q}$ for as long as possible, and finally append the remaining elements from either $\sigma_{p,i-1}$ or $\sigma_{i+1,q}$. Since $\sigma_{p,i-1}$ has length $a_i$, and $\sigma_{i+1,q}$ has length $b_i$, we have $\lambda'(S_{p,q}, i, \sigma_{p,q}) \ge m_i + \min(a_i, b_i) \ge \frac{k}{2}$.
		
		Clearly, if $j \in [p,i-1]$, then $\depth_{S_{p,q}}(j) = 1 + \depth_{S_{p,i-1}}(j)$, and similarly for elements $j \in [i+1,q]$ in the right subtree. Thus, by induction,
		\begin{align*}
			c(p,q) &= m_i + a_i + c(p,i-1) + b_i + c(i+1,q)\\
				&\le k + 2\Lambda'(S_{p,i-1}, \sigma_{p,i-1}) + 2\Lambda'(S_{i+1,q}, \sigma_{i+1,q})\\
				&\le 2 \lambda'(S_{p,q}, i, \sigma_{p,q}) + 2\Lambda'(S_{p,i-1}, \sigma_{p,i-1}) + 2\Lambda'(S_{i+1,q}, \sigma_{i+1,q}) \le 2\Lambda'(S_{p,q}, \sigma_{p,q}).
		\end{align*}
	
		Now let $S = S_{1,n}$ and $\sigma = \sigma_{1,n}$. We have $c(1,n) \le 2 \Lambda'(S,\sigma)$, and by \cref{p:opt-st-entropy}, we know that $c(1,n) \ge \cost(m_1, m_2, \dots, m_n) \ge H(m_1, m_2, \dots, m_n)$. This concludes the proof.
	\end{proof}
	
	\subsection{Wilber's lower bound for rotation distance}
	
	We now show that the rotation distance between $A(S, \pi)$ and $B(S)$ is at least $\frac{1}{2} \Lambda'(S, \sigma(\pi))$, where $\sigma(\pi)$ is defined as in \cref{sec:upper}, by replacing the leaf nodes in $\pi$ with their adjacent spine nodes. Our proof is based on~\cite[Lemma 8]{CardinalEtAl2018b}.
	
	\begin{lemma}\label{p:wilber-stg}
		Let $G = C(m_1, m_2, \dots, m_n)$ be a caterpillar, let $S$ be a BST on the spine nodes of $G$, and let $\pi$ be an ordering of the leg nodes of $G$.
		Then, transforming $A(S,\pi)$ into $B(S)$ requires at least $\frac{1}{2}\Lambda'(S, \sigma(\pi))$ rotations.
	\end{lemma}
	\begin{proof}
		Write $T = A(S, \pi)$, $T' = B(S)$, $\sigma = \sigma(\pi)$, and let $r$ be the root of $S$. Let the set $D$ consist of $r$ and its adjacent legs, i.e., if $r = s_i$, then $D = \{s_i\} \cup \{\ell_{i,j} \mid j \in [m_i]\}$. Suppose $r$ has two children $u$ and $v$. Then $G \setminus D$ has two connected components, one consisting of the spine nodes $V(S_u)$ and all adjacent legs, and the other consisting of $V(S_v)$ and all adjacent legs. Call the former $E$ and latter $F$. If $r$ has only one child $u$, let $E$ consist of $V(S_u)$ and all adjacent legs, and let $F = \emptyset$. If $r$ has no children, let $E = F = \emptyset$. Note that $D,E,F$ form a partition of $V(G)$.
		
		We first consider the rotations within each of the three sets. By \cref{p:proj}, we can simply sum up the number of rotations required to transform $T[D]$, $T[E]$, and $T[F]$ into $T'[D]$, $T'[E]$, $T'[F]$, respectively.
		
		$T[D]$ consists of the spine node $r$ and $m_r$ free leg nodes, and $T'[D]$ consists of $r$ and $m_r$ bound leg nodes. Thus, we need $m_r$ rotations to make all leg nodes bound.
		
		If $E \neq \emptyset$, observe that $T[E] = A(S_u, \pi\restr{E})$ and $T'[E] = B(S_u)$, so we need $\frac{1}{2}\Lambda'(S_u, \sigma\restr{E}) = \frac{1}{2}\Lambda'(S_u, \sigma)$ rotations by induction. If $F \neq \emptyset$, we similarly get a lower bound of $\frac{1}{2}\Lambda'(S_v, \sigma)$.
		
		We now show that there are at least $\frac{1}{2}\lambda(S,r,\sigma)$ other rotations. If $\lambda(S,r,\sigma) = 0$, this is trivially true, so suppose $\lambda(S,r,\sigma) > 0$ and thus $E \neq \emptyset$.
		
		Define the \emph{alternation number} of a path $P$ in a search tree on $G$ as the number of edges $(x,y)$ in $P$ such that $x$ and $y$ are in different parts of the partition $D,E,F$. Define the alternation number $\alt(T^*)$ of a search tree $T^*$ on $G$ as the maximum alternation number among all paths starting at the root in $T^*$. Observe that $\alt(T') = 1$, and that $\alt(T) \ge \lambda(S,r,\sigma)+1$, since the leg nodes in $T$ have $\lambda(S,r,\sigma)$ alternations by definition, and there is one more alternation from $r$ to $E \neq \emptyset$.
		
		We now show how rotations can affect the alternation number. Consider a rotation between the nodes $x$ and $y$, and a node $z$. The path from the root to $z$ before and after the rotation may only differ if it contains $x$ or $y$ (or both), and only in one of the following ways:
		\begin{itemize}
			\itemsep0pt
			\item $x$ is inserted before $y$, or $y$ is inserted before $x$.
			\item $x$ is deleted before $y$, or $y$ is deleted before $x$.
			\item $x$ and $y$ are swapped (and are neighbors).
		\end{itemize}
		
		It is easy to see that if $x,y \in D$, or $x,y \in E$, or $x,y \in F$, then the rotation $(x,y)$ does not affect the alternation number, and otherwise, it can only differ by at most two. This means that we need at least $\frac{1}{2}|\alt(T) - \alt(T')| \ge \frac{1}{2}\lambda(S,r,\sigma)$ rotations not within one of the sets $D$, $E$, or $F$. The total number of rotations is thus at least (setting $\Lambda'(S',\sigma) = 0$ if $S'$ is empty):
		\begin{align*}
			m_r + \frac{1}{2}\Lambda'(S_u, \sigma) + \frac{1}{2}\Lambda'(S_v, \sigma) + \frac{1}{2}\lambda(S,r,\sigma) \ge \frac{1}{2}\Lambda'(S, \sigma).\tag*{\qedhere}
		\end{align*}
	\end{proof}
	
	\Cref{p:wilber-static-opt,p:wilber-stg} together imply that $\delta(\fA(G)) \ge \frac{1}{4} H(G) \cdot m$. As mentioned in the introduction, Manneville and Pilaud~\cite{MannevillePilaud2015} proved that $\delta(\fA(G)) \in\Omega(m+n)$. This concludes the proof of the lower bound.
	
	\section{Conclusion}
	
	In this paper, we determined the diameter of each caterpillar associahedron up to a constant, revealing a surprising connection to searching in static and dynamic binary search trees. In particular, transforming $A(S, \pi)$ into $B(S)$ via rotations can be seen as a generalization of serving the access sequence $\sigma(\pi)$ in a dynamic BST. The number of rotations required is between Wilber's first lower bound $\Lambda(S, \sigma(\pi))$ and $\OPT(S, \sigma(\pi))$, begging the question whether other lower bounds for $\OPT$ hold in our generalized model, or whether it perhaps matches $\Lambda$ or $\OPT$. Results in this direction could give new insight into the dynamic BST model.
	
	\paragraph{Acknowledgment.} The author would like to thank L\'aszl\'o Kozma for helpful discussions and suggestions.
	
	\bibliography{info}
	\bibliographystyle{halphashort}
\end{document}